\documentclass{amsart}

\usepackage{color}
\usepackage{amsmath}
\usepackage{amsfonts}
\usepackage{amssymb,enumerate}
\usepackage{enumitem}
\usepackage{amsthm}
\usepackage{comment}
\usepackage[all]{xy}
\usepackage[utf8]{inputenc}
\usepackage[noadjust]{cite}
\usepackage{lineno}
\usepackage[autostyle]{csquotes}

\theoremstyle{definition}
\newtheorem{lem}{Lemma}[section]
\newtheorem{por}[lem]{Porism}

\newtheorem{prop}[lem]{Proposition}
\newtheorem{thm}[lem]{Theorem}

\newtheorem{defn}[lem]{Definition}
\newtheorem{ex}[lem]{Example}

\newcommand{\MO}{\mathcal{O}}

\newcommand{\bbz}{\mathbb{Z}}
\newcommand{\bbn}{\mathbb{N}}
\newcommand{\bbq}{\mathbb{Q}}

\newcommand{\bbf}{\mathbb{F}}

\newcommand{\imp}{\ensuremath{\Rightarrow{}}}
\newcommand{\nimp}{\ensuremath{\Leftarrow{}}}

\renewcommand{\geq}{\geqslant}
\renewcommand{\leq}{\leqslant}

\numberwithin{equation}{section}

\begin{document}
\author{Jared Kettinger}

\title{A generalized Davenport constant of the second kind}

\begin{abstract}

 In this paper, we explore a ring invariant which is closely related to the Davenport constant of a group. In particular, we will calculate this invariant for a certain class of rings of integers and their orders and use it to understand factorization properties of the latter. To this end, we also examine the well-behaved class of Galois-invariant orders.

\end{abstract}

\maketitle
\sloppy

\section{Introduction and notation}

The Davenport constant was first introduced by H. Davenport at a conference on group theory and number theory in 1966 (\cite{davenport1966midwestern}). Denoted $D(G)$ for a finite abelian group $G$, let us recall the following two equivalent definitions of this constant.

\begin{defn}
        max\{$n$ $\vert$ $\exists$ a zero-sum sequence of length $n$ with no non-empty proper zero-sum subsequence\}
\end{defn}

\begin{defn}
       min\{$n$ $\vert$ any $G$-sequence of length $n$ must have a non-empty zero-sum subsequence\} 
\end{defn}

While this constant has combinatorial interest in its own right, having been studied in this capacity by the likes of Erd{\"o}s (\cite{bovey1975conditions}) and Olson (\cite{olson1969combinatorial}), it was first posited due to its connection with factorization. In particular, if $G$ is the class group of the ring of integers $\MO_K$, and $\alpha \in \MO_K$ an irreducible element, $D(G)$ is the maximum number of prime ideals (counting multiplicity) in the factorization of the principal ideal $\alpha \MO_K$. While rings of integers fail to be unique factorization domains in general, they are Dedekind, meaning all nonzero proper ideals factor uniquely into prime ideals. Thus, understanding the prime factorization of a principal ideal $\beta \MO_K$ should help us understand the factorization of the element $\beta \in \MO_K$. The Davenport constant plays a central role in elucidating this connection. 

Now, for a given atomic domain, we would like a way to characterize its factorization. With this in mind, we recall the following definition of elasticity. 

\begin{defn}
    Let $R$ be an atomic domain. The \textit{elasticity} of a nonzero, nonunit $r \in R$ is defined as 
    \[ 
    \rho(r) = sup\Bigl\{\frac{n}{m} \,|\, r = \alpha_1 \alpha_2 \dotsm \alpha_n = \beta_1 \beta_2 \dotsm \beta_m\Bigr\} 
    \]
     where $\alpha_i, \beta_j \in Irr(R)$ for all $i,j$.

    Similarly, the \textit{elasticity} of the domain $R$ is defined as 
    \[ 
    \rho(R) = sup\Bigl\{\frac{n}{m} \,|\, \alpha_1 \alpha_2 \dotsm \alpha_n = \beta_1 \beta_2 \dotsm \beta_m\Bigr\}
    \]
     where $\alpha_i, \beta_j \in Irr(R)$ for all $i,j$. 
\end{defn}

 Elasticity essentially serves as a measure of the failure of unique factorization in a domain---in particular, unique factorization length. It has long been known that the class number of a ring of integers also serves as such a measure. For example, a ring of integers is a unique factorization domain (UFD) if and only if it has class number $h = 1$. Carlitz (\cite{Carlitz1960}) further showed that a ring of integers $R$ is a half-factorial domain (HFD), meaning $\rho(R) = 1$, if and only if $ h \leq 2$. The direct relationship between the Davenport constant and factorization is best illustrated by the following result due to Narkiewicz (\cite{Narkiewicz1995Note}) and Valenza (\cite{valenza1990elasticity}). 
\begin{thm}[\cite{Narkiewicz1995Note}]
    Let $R$ be a ring of integers with class number $h > 1$. Then
    \[
    \rho(R) = \frac{D(Cl(R))}{2}.
    \]
\end{thm}

Note that if $h =1$, $R$ is a UFD and thus $\rho(R)$ = 1.

Now, a natural progression from this result would be to study the factorization of orders $\MO$ in the rings of integers $\MO_K$. However, any proper order of a ring of integers fails to be Dedekind. In particular, only those prime ideals ideals prime to the conductor $\mathfrak{f}_\MO = \{\alpha \in \MO_K | \alpha \MO_K \subseteq \MO\}$ are invertible in $\MO$ (\cite{cox2022primes}), so the Davenport constant of the generalized class group can only help us directly compute the elasticity of certain elements. Hence, we need to develop some more sophisticated tools to study factorization in these orders entirely. We would like to find a way to use our knowledge of factorization in the ring of integers above to give us insight into factorization in the order below. Just as Davenport leverages the factorization of the ideal $\beta \MO_K$ to understand the factorization of the element $\beta \in \MO_K$, we seek to leverage the factorization of the element $\beta$ in $\MO_K$ to understand its factorization in an order $\MO \subseteq \MO_K$. More generally, we might ask: how can we use our knowledge of factorization in a ring $R$ to gain knowledge about factorization in a given subring?

Let us consider the simplest case where $\mathcal{O} \subseteq R$ is an integral extension of atomic domains with $R$ a UFD and $U(R) = U(\MO)$. Then, for any nonzero, nonunit $x \in \mathcal{O}$, we may consider it also as an element of $R$ where it enjoys a unique factorization into prime elements:

\begin{center}
    $x = p_1 p_2 \dotsm p_n$.
\end{center}

Now, if $x = \pi_1 \pi_2 \dotsm \pi_k$ is an irreducible factorization of $x$ in $\mathcal{O}$, by uniqueness, each 
\[
\pi_i = \prod_{j \in A \subseteq [1,n]}p_j
\]
\noindent up to a unit multiple. Hence, if we can determine how elements (in particular, prime elements) of $R$ combine to produce elements of $\mathcal{O}$, we can completely determine the factorization properties of $\mathcal{O}$. 

\begin{ex}
    Ignoring for the moment the discrepancy in their unit groups, consider the Gaussian integers $\bbz[i]$ and the order $\bbz[7i] \subseteq \bbz[i]$. $\bbz[i]$ is a unique factorization domain, so the element $490$ factors \textit{uniquely} into primes as $490 = 7\cdot7 (1+i)(1-i)(1+2i)(1-2i)$. Now, we want to group these elements to form irreducibles of $\bbz[7i]$. One possible grouping is $[7][7][(1+i)(1-i)][(1+2i)(1-2i)] = 7\cdot7\cdot2\cdot5$. Another option is $[7(1+i)(1-2i)][7(1-i)(1+2i)] = (21-7i)(21+7i)$. The fact that these elements are all irreducible in $\bbz[7i]$ follows from the fact that the prime factorization in $\bbz[i]$ is unique, and there is no way to non-trivially partition the primes we have grouped together to give a finer factorization in $\bbz[7i]$. For example, looking at $[7(1+i)(1-2i)]$, $(1+i)(1-2i) = 3 - i \notin \bbz[7i]$. Also, despite $7(1+i)$ and $7(1-2i)$ being in $\bbz[7i]$, this leaves us with a factor which is not, namely $1+i$ or $1-2i$. Thus, we have irreducible factorizations of lengths 2 and 4, so $\rho(490) \geq \frac{4}{2} = 2$ in $\bbz[7i]$. 
\end{ex}

With this example in mind, we present the following definitions. 

\begin{defn}
    Let $R$ be an atomic monoid and $\mathcal{O} \subseteq R$ a non-empty subset. We say an $R$-product $\pi_1 \pi_2 \dotsm \pi_n$ of elements $\pi_i \in \text{Irr}(R)$ is an $\mathcal{O}$\textit{-product} if $\pi_1 \pi_2 \dotsm \pi_n \in \mathcal{O}$.
\end{defn}

\begin{defn}
    Let $R$ be an atomic monoid and $\mathcal{O} \subseteq R$ a non-empty subset.

    \begin{center}
        $D_\MO(R) =$ min $\{ n \,\vert$ any $R$-product of length $n$ has an $\mathcal{O}$-subproduct$\}$
    \end{center}

    If  no such $n$ exists, we say $D_\MO(R) = \infty$. In the case that $\MO$ is a domain and $R$ its integral closure, we will use $\Bar{D}(\MO)$ to denote $D_\MO(R)$.

\end{defn}

We will refer to this as the generalized Davenport constant of $R$ over $\mathcal{O}$. In this paper, we will focus almost exclusively on the case when $R$ is a domain. Now, the use of $\mathcal{O}$ to denote the subset is suggestive in two ways. First, it hints that our primary objects of study will be orders of a ring of integers -- in which case we will make extensive use of the $\Bar{D}(\MO)$ notation. Second, it is reminiscent of the fact that the subset $\mathcal{O}$ is taking the place of $0$ in the definition of the Davenport constant. 

Returning to finite abelian groups, we can make a similar generalization. For an abelian group $G$ and subgroup $H \leq G$, we call a $G$-sequence $g_1 g_2\dots g_n$ an $H$-sum sequence if $g_1 + g_2 + \dotsm + g_n \in H$. Then we may allow $D_H(G)$ to denote the minimum $n$ such that any $G$-sequence of length $n$ must have an $H$-sum subsequence. This coincides with the study of $\mathcal{B}_H(G)$, the monoid of $H$-sum sequences, introduced by Halter-Koch in \cite{halter1992relative}. In our notation, we have the nice symmetry $D_H(G) = D(G/H)$ which can be found in \cite{baeth2009atoms}. We will see that \textit{quotient rings} also play an important role in understanding $D_\mathcal{O}(R)$. 

Now, our definition of $D_{\MO}(R)$ clearly echoes the second definition of $D(G)$: min\{$n$ $\vert$ any $G$-sequence of length $n$ must have a non-empty zero-sum subsequence\}. Thus, it is natural to ask if $D_\MO(R)$ is equal to sup$\{n$ $\vert$ $\exists$ an $\MO$-product of length $n$ with no non-empty proper $\MO$-subproduct$\}$ in parallel with $D(G)$. The following example demonstrates that these values are distinct.

\begin{ex}
    Consider the ideal $5\bbz[i] = (5)  \subseteq \bbz[i]$ which factors into prime ideals as $(2 + i)(2 - i)$. Hence, the element $(2 + i)^n \notin (5)$ for any $n \in \bbn$. Thus, we have $D_{5\bbz[i]}(\bbz[i]) = \infty$. However, given any product $\pi_1 \pi_2 \dotsm \pi_n \in 5\bbz[i]$ with $\pi_i \in \text{Irr}(R)$, as $\bbz[i]$ is a UFD and $5$ divides the product, we may assume without loss of generality that $\pi_1 = 2 + i$ and  $\pi_2 = 2 - i$, so $\pi_1 \pi_2 = 5 \in 5\bbz[i]$. Note, this is technically up to a $U(\bbz[i])$ multiple of $5$, but these all lie in $5\bbz[i]$. Therefore, the longest $5\bbz[i]$-product with no proper $5\bbz[i]$-subproduct has length $2$.
\end{ex}

This example is also significant as it distinguishes this generalized Davenport constant from the omega invariant. First introduced by Geroldinger in \cite{geroldinger1997chains}, for a integral domain $R$ and nonzero nonunit $a \in R$, $\omega_R(a)$ denotes the smallest $n \in \bbn$ such that for any $b_1,b_2,...,b_{n+1} \in R$ such that $a$ divides $\prod_{i=1}^{n+1}b_i$, there exists a non-empty, proper subset $S \subset \{1,2,...,n+1\}$ such that $a$ divides $\prod_{i\in S}b_i$. If no such $n$ exists, we write $\omega_R(a) = \infty$. Subsequently, for an atomic domain $R$, we define $\omega(R) = \text{sup}\{\omega(a) \,|\, a \text{ is an irreducible element of } R\}$. For an atomic domain $R$ and a proper ideal $I \subseteq R$, we have the related definition: $\omega(I)$ is the smallest $n \in \bbn$ such that for any $a_1,a_2,...,a_{n+1} \in R$ with $\prod_{i = 1}^{n+1} a_i \in I $ there exists a non-empty, proper subset $S \subset \{1,2,...,n+1\}$ such that $\prod_{i \in S} a_i \in I$. If no such $n$ exists, we write $\omega(I) = \infty$. Now, recalling Example 1.8, we see that $D_{5\bbz[i]}(\bbz[i]) = \infty$ while $\omega(5\bbz[i]) = 2$. 

In many ways, the omega invariant is reminiscent of the first definition of the Davenport constant: max\{$n$ $\vert$ $\exists$ a zero-sum sequence of length $n$ with no non-empty proper zero-sum subsequence\}. However, the two invariants cannot be reconciled because, unlike in the group setting, given an $R$-product $a_1a_2\dotsm a_n$ with no $\MO$-subproduct, there is no guarantee that there exists $a_{n+1} \in R$ such that $a_1a_2\dotsm a_{n+1}$ is an $\MO$-product with no proper $\MO$-subproduct. 

Now, the omega invariant has long been studied in connection with factorization---see for example \cite{geroldinger2006non} and \cite{geroldinger2010arithmetic}. Concurrently with this work, Choi made extensive use of the omega invariant to study the factorization properties of orders of a PID (\cite{choi2024class}). As we will see, the distinctions between $\omega(R)$ and $D_{\MO}(R)$ result in distinct techniques and constructions. As with the omega invariant, some of the definitions and constructions developed here have ideal-theoretic generalizations. These are explored in part by Moles and the current author in \cite{kettinger2025elasticity} to determine properties of arbitrary orders of rings of integers with prime conductor.

Before moving on, let us remark on a few other generalizations of the Davenport constant. The large and small Davenport constants, defined for a finite group as $D(G) = \text{max}\{n \, | \, \exists $ a zero-sum sequence with no proper non-empty zero-sum subsequence$\}$ and $d(G) = \text{max}\{m\,|\, \exists $ a $G$-sequence of length $m$ with no non-empty zero-sum subsequence$\}$ respectively, were introduced and studied in \cite{geroldinger2013large} and \cite{grynkiewicz2013large}. These definitions similarly attempt to account for the discrepancy between the two definitions of $D(G)$ when $G$ is no longer assumed to be abelian. In many ways, the omega invariant represents a generalization of the large Davenport constant as $D_\MO(R)$ represents a generalization of the small davenport constant. For a more detailed exploration of these concepts, the interested reader is encouraged to see \cite{cziszter2016interplay}. Finally, the term \enquote{generalized Davenport constant} is sometimes used in reference to the $k$th Davenport constant $D_k(G) = \text{min}\{l\,|\, \text{any } G$-sequence of length $l$ has at least $k$ disjoint non-empty zero-sum subsequences$\}$ first defined in \cite{halter1992generalization}.

Returning to the definition in question, when $\MO$ is a subring, we will find that $D_\MO(R)$ and $D(G)$ are not only similarly defined, but in many cases $D_\MO(R)$ is intimately connected with the Davenport constant of a group related to $\MO$. As previously mentioned, we will primarily be concerned with orders contained in rings of integers. Throughout the rest of the paper, we will use $\MO$ to denote an order of a ring of integers $\MO_K$ and $\mathfrak{f}$ to denote the conductor of this extension. In section 2, we will study how $D_\MO(\MO_K) = \Bar{D}(\MO)$ relates to the elasticity of $\MO$. In section 3, we will apply results and techniques developed in section 2 to prove some stronger results for a class of orders in quadratic rings of integers.

\section{Conditions for finite elasticity of orders in rings of integers}

Intuitively, if $\Bar{D}(R)$ is finite and $\Bar{R}$ is well-behaved, we would expect the same of $R$ to some extent. In this section, we explore some conditions under which $\Bar{D}(R)$ gives us information about the finiteness of $\rho(R)$ and vice versa. To aid in this process, we will investigate a class of well-behaved orders in rings of integers. The following lemma will play an integral part in exploring the relationship between these two constants. 

\begin{lem}
    Assume $\MO_K$ has class number $1$. If $\alpha_1, \alpha_2, \dots, \alpha_m$ are irreducibles of $\MO_K$ and $r \in \mathfrak{f}$ an irreducible of $\MO_K$, then $r \alpha_1 \alpha_2 \dotsm \alpha_m$ is irreducible in $\MO$ if and only if $\alpha_1 \alpha_2 \dotsm \alpha_m$ has no $\MO$-subproduct up to multiplication by units of $\MO_K$.
\end{lem}

\begin{proof}
    $(\imp\!\!)$ Proceeding by contraposition, we will assume without loss of generality that there exists $u \in U(\MO_K)$ such that $u \alpha_1 \alpha_2 \dotsm \alpha_t \in \MO$ for some $1 \leq t \leq m$. Thus, $r \in \mathfrak{f}$ implies $r \alpha_1 \alpha_2 \dotsm \alpha_m = (u \alpha_1 \alpha_2 \dotsm \alpha_t)(r(u^{-1}\alpha_{t+1} \dotsm \alpha_m))$ is a non-trivial factorization in $\MO$.

    $(\nimp\!\!)$ We once again proceed by contraposition and assume $r \alpha_1 \alpha_2 \dotsm \alpha_m$ is not irreducible in $\MO$. Then, there exist nonzero, nonunits $x,y \in \MO$ such that $r \alpha_1 \alpha_2 \dotsm \alpha_m = xy$. The result now follows from the fact that $\MO_K$ is a unique factorization domain.
\end{proof}

This lemma will prove vital to finding elements of maximal elasticity in orders of number fields. The requirement that $\alpha_1 \dotsm \alpha_m$ has no $\MO$-subproduct \textit{up to units of $\MO_K$} is a subtle but important condition. We will explore this idea further in section 3.

Before moving on, let us recall some important theory of orders in number fields. Any order $\MO$ properly contained in the ring of integers $\MO_K$ fails to be Dedekind as it fails to be integrally closed. However, ideals relatively prime to the conductor $\mathfrak{f}_\MO = \{\alpha \in \MO_K \mid \alpha \MO_K \subseteq \MO\}$ are still invertible.  In many ways, these ideals and their elements behave similarly to those in the ring of integers above. We will leverage this fact often throughout the rest of the paper. Notably, we can form a \textit{generalized} class group $Cl(\MO)$ by taking the fractional ideals relatively prime to $\mathfrak{f}$ modulo the principal ideals relatively prime to $\mathfrak{f}$ (\cite{cox2022primes}). This leads us to the following theorem.

\begin{thm}[\cite{cohn1980advanced}]
    Let $\bbz[\omega]$ be the ring of integers of the number field $\bbq(\sqrt{d})$, then for $n > 1$, 
    \[
    |Cl(\bbz[n\omega])| = h(d)\psi_n(n)/u
    \]
    where $h(d)$ is the class number of $\bbz[\omega]$, $u = (\bbz[\omega]^* : \bbz[n\omega]^*)$, and $\psi_d(n) = n \underset{q \mid n}{\prod}(1- (d/q)/q)$.
\end{thm}

Significantly, we note that this implies $|Cl(\bbz[n\omega])|$ is finite. What is more, by Theorem 12.12 in \cite{neukirch2013algebraic}, the class group of an order in any number field is always finite. We are now prepared to explore one direction of the relationship between $\Bar{D}(\MO)$ and $\rho(\MO)$.

\begin{thm}
    Let $\MO_K$ be a ring of integers with class number 1, and $\mathcal{O} \subseteq \MO_K$ an order with conductor prime in $\MO_K$. If $\Bar{D}(\mathcal{O})$ is finite, then $\rho(\mathcal{O})$ is finite.
\end{thm}

\begin{proof}
    Assume for the purpose of contradiction that $\rho(\mathcal{O}) = \infty$. Then, for all $n \in \bbn$, there exist $\alpha_1, \alpha_2, \dots, \alpha_k, \delta_1, \delta_2, \dots, \delta_N \in$ Irr($\mathcal{O}$) where $N \geq nk$ such that  
    \[ 
    \alpha_1 \alpha_2 \dotsm \alpha_k = \delta_1 \delta_2 \dotsm \delta_N .
    \]
    As $\MO_K$ is the integral closure of $\mathcal{O}$, each $\delta_j$ remains a nonunit in $\MO_K$, and thus admits an irreducible factorization in the overring. In fact, as $\MO_K$ is a unique factorization domain, each $\delta_j$ factors into a product of primes in $R$. Let $p_1 p_2 \dotsm p_s$ be the prime factorization of $\delta_1 \delta_2 \dotsm \delta_N$ in $\MO_K$. By the above argument, we must have $s \geq N \geq nk$. As this factorization is unique, it must be the case that some $\alpha_i$ factors into a product of at least $n$ primes in $\MO_K$. Without loss of generality, $\alpha_1 = p_1 p_2 \cdots p_m$ where $m \geq n$. Thus, we have shown that an irreducible of $\mathcal{O}$ can factor into an arbitrarily long product of primes in $\MO_K$.

    Now, $\MO_K$ is a Dedekind UFD  and thus a PID. As the conductor of $\mathcal{O}$ is a prime ideal of $\MO_K$, it is of the form $(p)$ where $p \in \MO_K$ is prime. Assume for the purpose of contradiction that $p$ does not divide $\alpha_1$ in $\MO_K$. Then, $(\alpha_1) \subseteq \mathcal{O}$ is relatively prime to the conductor and hence has a prime factorization corresponding to $(\alpha_1) =(p_1)(p_2)\dotsm (p_m) \subseteq R$ (\cite{cox2022primes}). Say
    \[
    (\alpha_1) = \mathfrak{P}_1 \mathfrak{P}_2\dotsm \mathfrak{P}_m \subseteq \mathcal{O}
    \]
    where the $\mathfrak{P}_i$'s are not necessarily distinct elements of $Spec(\mathcal{O})$. As $m$ may be arbitrarily large and $D(Cl(\MO))$ is finite, we may assume $m >  D(Cl(\MO))$. However, $\alpha_1$ irreducible in $\mathcal{O}$ implies $[\mathfrak{P}_1][\mathfrak{P}_2] \dots [\mathfrak{P}_m]$ is a minimal zero-sum sequence in $Cl(\MO)$ which implies $D(Cl(\MO)) \geq m$, contradicting our previous statement. Hence, without loss of generality, we may assume $p_1 = p$. 

    Finally, as $p$ is in the conductor of $\MO$ and $\alpha_1 = p p_2 p_3\dotsm p_m $ is irreducible in $\MO$, by Lemma 2.1, $p_2 p_3 \dotsm p_m$ has no $\mathcal{O}$-subproduct. However, as $m$ is arbitrarily large, this implies $\Bar{D}({\mathcal{O}}) = \infty$, so we must have $\rho(\mathcal{O}) <  \infty$.

\end{proof}

The following result follows directly from the preceding proof.

\begin{por}
    Let $\MO_K$ be a ring of integers with class number 1, and $\mathcal{O} \subseteq \MO_K$ an order with conductor prime in $\MO_K$. Then
    \[
    \rho(\MO) \leq \text{max}\{D(Cl(\MO)), \Bar{D}(\MO)\}.
    \]
\end{por}

 Already, we are seeing the interconnectedness of these three constants. We should also be encouraged by the symmetry between this corollary and Theorem 1.4. Now, the following example demonstrates the significance of the prime conductor assumption. 

\begin{ex}
    Let $R = \bbz[i]$ and consider the orders $\bbz[3i]$ and $\bbz[5i]$. It is easy to see that these orders have conductors $(3)$ and $(5)$ respectively. Now, $3 \equiv 3$ mod $4$ implies it remains prime in $\bbz[i]$ while $5 \equiv 1$ mod $4$ implies that $5$ splits. In fact, as an element, $5 = (2 - i)(2 + i)$. Theorem 2.3 thus implies that $\rho(\bbz[3i]) < \infty$. In section 3, we will show that $\rho(\bbz[3i]) = \frac{3}{2}$. Contrastingly, note that $(2-i)^n, (2+i)^n \notin \bbz[5i]$ for any $n \in \bbn$. This is easiest to see by considering the coefficients modulo $5$. Thus, we have $(5(2-i)^n)\cdot(5(2+i)^n) = 5^{n+2}$ as irreducible factorizations in $\bbz[5i]$ of lengths $2$ and $n + 2$. Thus, $\rho(\bbz[5i]) \geq \frac{n +2}{2}$ for all $n \in \bbn$, so $\rho(\bbz[5i]) = \infty$.
\end{ex}

Once again, we note the theme of using products with no $\MO$-subproduct to form elements of large elasticity. We can also see that the structure of the conductor plays a vital role in determining factorization properties of the order. This relationship is made most explicit by Halter-Koch in \cite{halter1995elasticity}. In general, if there are multiple prime ideals of $\MO_K$ containing $\mathfrak{f}\MO$, we can form an element of arbitrarily large elasticity in $\MO$.  Similarly, Choi (\cite{choi2024class}) shows that factorization is well-behaved for orders of PIDs with primary conductor. We will explore the explicit calculation for some such orders in section 3. 

We have now seen how our generalized Davenport constant can be used to form elements of maximal elasticity and circumstances under which finiteness of $\Bar{D}(\MO)$ implies finiteness of $\rho(\MO)$. We would naturally like to know if and when the converse of Theorem 2.3 is true. Namely, when does finite elasticity imply $\bar{D}(\MO)$ is finite? Equivalently, when does $\bar{D}(\MO) = \infty$ imply $\rho(\MO) = \infty$? What properties of quadratic orders are we taking advantage of to create these elements of large elasticity? In effort to explore these questions, we present the following definition.

\begin{defn}
    An order $\MO$ of a Galois number field $K$ is called \textit{Galois-invariant} if it is closed under Galois conjugation. That is, for any $\alpha \in \MO$ and $\sigma \in \text{Gal}(K/\bbq)$, $\sigma (\alpha) \in \MO$. 
\end{defn}

This term was first formally used in in 2015 by Lee and Louboutin in \cite{lee2015determination} where the authors considered the family of Galois number fields generated by a totally real cubic algebraic unit. More recently, Geroldinger et al. made connection between Galois-invariance and factorization theory in \cite{geroldinger2022monoids}--investigating the factorization properties of the normsets of Galois-invariant orders. It appears that work on Galois-invariant orders qua algebraic objects remains relatively sparse in the literature. We will soon see the relevance of this definition to our exploration of $\Bar{D}(\MO)$, but for now, let us consider a few classes of Galois-invariant orders.

\begin{prop}
    Let $K$ be a Galois number field. The following orders are Galois-invariant.

    \begin{enumerate}
        \item The ring of integers $\MO_K$.

        \item If $K$ is quadratic, any order of $K$.

        \item Any order generated over $\bbz$ by a normal basis of $K$.
    \end{enumerate}
\end{prop}

First, for any $\beta \in \MO_K$, $\sigma(\beta)$ is also an algebraic integer contained in $K$, so we must have $\sigma(\beta) \in \MO_K$. All quadratic number fields are of the form $K = \bbq(\sqrt{d})$ for some square-free $d \in \bbz$. It is well known that every order of $K$ is of the form $\bbz[n\omega]$ where $n \in \bbn$ and 
\[
    \omega = \begin{cases} 
      \sqrt{d}  & d \equiv 2,3 \textit{ mod } 4 \\
      \frac{1+\sqrt{d}}{2} & d \equiv 1 \textit{ mod } 4. \\
   \end{cases}
\]
Now, the only elements of Gal$(K/\bbq)$ are the identity and complex conjugation, so in order to show $\bbz[n\omega]$ is Galois-invariant, it suffices to show $n \Bar{\omega} \in \bbz[n\omega]$. To this end, we note that $tr(n\omega) = n\omega + n \Bar{\omega} = m \in \bbz$ which implies $n \Bar{\omega} = m - n \omega \in \bbz[n \omega]$ as desired. 

Finally, recall that a normal basis of a $K$ (over $\bbq$) is a set of the form $\{\sigma(\alpha) | \sigma \in \text{Gal}(K/\bbq)\}$ for some $\alpha \in K$ which forms a basis of $K$ as a vector space over $\bbq$. Essentially, it is a basis formed by Galois conjugates. Hence, it is clear that any order of the form $\bbz[ \alpha_1, \alpha_2, \dots, \alpha_n]$ where $\{\alpha_1, \alpha_2, \dots, \alpha_n\}$ forms a normal (integral) basis, is Galois-invariant. Recall that the normal basis theorem guarantees any Galois number field admits a normal basis. Also, every order contains an integral basis of the extension $[K: \bbq]$. Unfortunately, the existence of a normal \textit{integral} basis is not guaranteed in general, but in the case $K$ does admit such a basis, this implies the existence of many Galois-invariant orders in the form of (3).

\begin{ex}
    For any rational prime $p$, the $p$-th cyclotomic number field $\bbq(\zeta_p)$ has normal integral basis $\{\zeta_p, \zeta_p^2, \dots, \zeta_p^{p-1}\}$ (\cite{narkiewicz1974elementary}). Then, if the greatest common divisor of $\{n_1, n_2, \dots, n_{p-1}\}$, with $n_i \in \bbn$, is a member of the set, the order $\bbz[ n_1 \zeta_p, n_2 \zeta_p^2, \dots, n_{p-1}\zeta_p^{p-1}] $ is Galois-invariant.
\end{ex}

This leads us to a natural question: are all Galois-invariant orders generated by a normal basis? The following example shows that this is not the case. 

\begin{ex}
    Consider the Gaussian integers $\bbz[i] \subseteq \bbq(i)$. Both 1 and 2 imply that $\bbz[i]$ is Galois-invariant. However, it is well known that the Gaussian field does not admit a normal integral basis. Additionally, recalling that an order of a number field $K$ is a $\bbz$-module of rank $[K:\bbq]$, one can observe directly that $\bbz[i]$ cannot be expressed in the form $\langle a + bi, a- bi \rangle_\bbz$.
\end{ex}

This tells us that there are \enquote{non-obvious} examples of Galois-invariant orders. We might even begin to wonder if all orders in Galois number fields are Galois-invariant. This turns out not to be the case.

\begin{ex}
    Consider the Galois number field $\bbq(\sqrt[6]{3}i) = \bbq(\alpha)$ and the order $\bbz[\alpha] = \langle 1,\alpha, \alpha^2, \dots ,\alpha^5\rangle_\bbz$. Observe that this order does not contain the conjugate $\frac{\alpha}{2} + \frac{\alpha^4}{2} = \alpha \cdot \omega$ (where $\omega = \frac{1}{2}-\frac{\sqrt{3}}{2}i$ is a $6^{th}$ root of unity) and thus is not Galois-invariant. Additionally, we note that $\bbz[\alpha] = \bbz[\alpha, 2\cdot \sigma_1(\alpha), 2\cdot \sigma_2(\alpha), \dots, 2 \cdot \sigma_5(\alpha)]$, which motivates this example. Many non-invariant orders in number fields of degree $3$ or greater can be constructed in this manner. 

\end{ex}

With this concept in hand, we are now prepared to return to the question of when finite elasticity implies $\bar{D}(\MO)$ is finite.

\begin{thm}
    Let $K$ be a Galois number field with class number $1$, and $\MO$ a Galois-invariant order of $K$ with prime conductor such that $U(\MO) = U(\MO_K)$. Then if $\rho(\MO)$ is finite, $\Bar{D}(\MO)$ is finite. In particular, $\Bar{D}(\MO) \leq n(\rho(\MO) - 1) + 1$ where $n = [K:\bbq]$.
\end{thm}

\begin{proof}
    Let $\alpha_1 \alpha_2 \dotsm \alpha_k$ be an $\MO_K$-product with no $\MO$-subproduct. Let us write $Gal(K/\bbq) = \{\sigma_1, \sigma_2, \dots, \sigma_n\}$. Now, $\MO$ Galois-invariant implies $\sigma_i(\alpha_1) \sigma_i(\alpha_2) \dotsm \sigma_i(\alpha_k)$ has no $\MO$-subproduct for any $\sigma_i \in Gal(K/\bbq)$. This is because $\sigma_i(\alpha_1)\sigma_i(\alpha_2)\dotsm \sigma_i(\alpha_t) = \sigma_i(\alpha_1 \alpha_2 \dotsm \alpha_t) \in \MO$ implies $(\sigma_i^{-1} \circ \sigma_i)(\alpha_1 \alpha_2 \dotsm \alpha_t) = \alpha_1 \alpha_2 \dotsm \alpha_t \in \MO$. Note that as $\MO_K$ is a PID, every prime ideal is generated by a prime element, so we may choose $r \in \mathfrak{f}$ prime in $\MO_K$. Thus, as $U(\MO) = U(\MO_K)$, by Lemma 2.1, $r \sigma_i(\alpha_1)\sigma_i(\alpha_2)\dotsm \sigma_i(\alpha_k)$ is irreducible in $\MO$ for all $1 \leq i \leq n$. Now, consider the following element of $\MO$
    \[
    \prod_{i = 1}^n (r \sigma_i(\alpha_1)\sigma_i(\alpha_2)\dotsm \sigma_i(\alpha_k)) = r^n \prod_{j=1}^k \prod_{i=1}^n \sigma_i(\alpha_j) = r^n \prod_{j=1}^k N(\alpha_j).
    \]

    Now, each $N(\alpha_j) \in \bbz \subseteq \MO$, and $\alpha_j \notin U(\MO_K)$ implies $N(\alpha_j) \notin U(\MO_K) = U(\MO)$ for all $1 \leq j \leq k$. Thus, we have $\rho(\MO) \geq \frac{n + k}{n}$, so $k$ is bounded above. Now, if $\rho(\MO) = 1$, we must have $k = 0$. Otherwise we can rearrange to achieve $k \leq n(\rho(\MO) - 1)$. This establishes the result. 
    \end{proof} 

\begin{ex}
    Consider the ring of integers $\bbz[\frac{1 + \sqrt{-7}}{2}] = \bbz[\omega]$ and the order $\bbz[5\omega]$. We will soon show that $\rho(\bbz[5\omega]) = \frac{7}{2}$. Thus, Theorem 2.11 implies $\Bar{D}(\bbz[5\omega]) \leq 2 \left( \frac{7}{2} - 1 \right) + 1 = 6$. Observe that the product $\omega^5$ has no $\bbz[5\omega]$ subproduct. Hence, we must have  $\Bar{D}(\bbz[5\omega]) = 6$.
\end{ex}

In the following section, we consider some classes of orders for which Theorems 2.3 and 2.11 both hold. Namely, orders in which $\Bar{D}(\MO)$ is finite if and only if $\rho(\MO)$ is finite. We will also develop some techniques for computing both constants in such orders.

\section{Elasticity of orders in quadratic number fields}

The results in the section are due to work done concurrently with but separate from Choi (\cite{choi2024class}). In particular, an equivalent of Theorem 3.6 can also be found in \cite{choi2024class} using a very distinct approach from this paper. This section also represents an extension of the theory developed in \cite{Kettinger2024}.

\begin{lem}[\cite{sittinger2016quotients}]
    Let $\bbz[\omega]$ be a quadratic ring of integers with class number $1$, and $p \in \bbz[\omega]$ a prime element. Then, for any $r \in \bbz[\omega]$ relatively prime to $p$, there exists a prime element $q \in \bbz[\omega]$ such that $q \equiv r$ mod $p$.
\end{lem}

In fact, much more than this is true. The prime elements of $\bbz[\omega]$ are asymptotically uniformly distributed among the invertible classes of $\bbz[\omega]/(p)$, but for our purposes, the existence of a single element in each of these classes will be sufficient. 

We are now ready to calculate $\Bar{D}(\MO)$ for a large class of orders. Recall once again that all orders of a quadratic number field are of the form $\MO = \bbz[n \omega]$ for any $n \in \bbz$. Hence, the conductor $\mathfrak{f} = n\MO_K$ is always a principal ideal of the ring of integers. This is a unique and desirable property of which we will make extensive use. The following theorem makes explicit the connection between $\mathfrak{f}_\MO$ and $\Bar{D}(\MO)$ hinted at previously. 

\begin{thm}
    Let $\bbq(\sqrt{d})$ be a quadratic number field with class number 1 and $d < -3$. If $p$ is a rational prime such that $\binom{d}{p} = -1$, then $\Bar{D}(\bbz[p\omega]) = p + 1$. 
\end{thm}

Note that the assumption $d < -3$ ensures $U(\bbz[p\omega]) = U(\bbz[\omega])$, and $\binom{d}{p} = -1$ if and only if $p$ is inert in $\bbz[\omega]$. 

\begin{proof}
    First, we note that as $(p) \subseteq \bbz[\omega]$ is prime, $\bbz[\omega] /(p)$ is a finite field of order $p^2$. If $\hat{\omega}$ is a root of $x^2 + d$, the defining polynomial of $\bbq(\sqrt{d})$, over the field $\mathbb{F}_p$, then the map $[a + b\omega] \overset{\phi}{\mapsto} \hat{a} + \hat{b}\hat{\omega}$, where the coefficients are reduced mod $p$, is a field isomorphism from $\bbz[\omega]/(p)$ to $\mathbb{F}_p [\hat{\omega}]$.

    Now, considering $\bbz[\omega]/(p)$, the elements of $\bbz[p\omega]$ comprise the cosets of $(p)$ of the form $[a + 0\cdot\omega]$, which have isomorphic image $\mathbb{F}_p$ under the map defined above. Thus, we can identify $\bbz[p\omega]\subseteq \bbz[\omega]$ with $\mathbb{F}_p \subseteq \mathbb{F}_p[\hat{\omega}]$. Reducing mod $p$, we see that a $\bbz[\omega]$-product $\pi_1\dotsm\pi_N$ has no $\bbz[p\omega]$-subproduct if and only if $\hat{\pi}_1\dotsm\hat{\pi}_n$ has no subproduct in $\mathbb{F}_p$. In other words,  $\Bar{D}(\bbz[p\omega])$ is less than or equal to the minimum $N$ such that any product of length $N$ in  $\mathbb{F}_p[\hat{\omega}]$ must have a subproduct in $\mathbb{F}_p$. Also, because $\mathbb{F}_p[\hat{\omega}]$ has no nonzero zero divisors, this is equivalent to finding the minimum $N$ such that any product of length $N$ in $\mathbb{F}_p[\hat{\omega}]^*$ must have a subproduct in $\mathbb{F}_p^*$.

    Now, $\bbf_p$ and $\bbf_p[\hat{\omega}]$ are finite fields, so $\bbf_p^*$ and $\bbf_p[\hat{\omega}]^*$ are finite cyclic groups with $\bbf_p^* \subseteq \bbf_p[\hat{\omega}]^*$. Hence, $G = \bbf_p[\hat{\omega}]^*/\bbf_p^*$ is a finite abelian group of order $\frac{p^2 -1}{p-1} = \frac{(p+1)(p-1)}{p-1} = p + 1$. Thus, finding the finding the minimum $N$ such that any product of length $N$ in $\mathbb{F}_p[\hat{\omega}]^*$ must have a subproduct in $\mathbb{F}_p^*$ is equivalent to finding the minimum $N$ such that any $G$-sequence has a zero-sum subsequence. Notice, this is just $D(G)$. As $G$ is a quotient group of a cyclic group, it is also cyclic, so $D(G) = |G| = p + 1 $. Thus, we have $\Bar{D}(\bbz[p \omega]) \leq p + 1$. 

    Once again, as $p \in \bbz[\omega]$ is prime, every nonzero class of $\bbz[\omega]/(p)$ is invertible. Let $r + (p)$ be a class whose image under the map $\phi$ mod $\mathbb{F}_p^*$ generates $G$. This must obviously be a nonzero class. Hence, by Lemma 3.1, we may choose $\gamma \in r + (p)$ prime, and by construction, $\gamma^p$ will have no $\bbz[p\omega]$-subsequence. Therefore, we conclude $\Bar{D}(\bbz[p\omega]) = p + 1$.

\end{proof}

Now, we not only know $\Bar{D}(\bbz[p\omega])$, but we have a procedure for finding a product achieving it. 

\begin{ex}
    Returning to a previous example, consider the ring of integers $\bbz[\frac{1 + \sqrt{-7}}{2}] = \bbz[\omega]$ and the order $\bbz[5\omega]$. Theorem 3.2 reaffirms that $\Bar{D}(\bbz[5 \omega]) = 6$. It also motivates our choice of $\omega^5$ as our maximal $\bbz[\omega]$-product with no $\bbz[5\omega]$-subproduct. In this case, as $|\bbf_5[\hat{\omega}]^*/\bbf^*| = 6$, we can see that one third of the non-zero equivalence classes mod $p$ will work for our purposes. To see that $\omega = \frac{1 + \sqrt{-7}}{2}$ works, by Lagrange's theorem, we need only check that $\omega^n \text{ mod } p \notin \bbf_p$ for $n \leq 3$, and in fact $\omega^2 = -2 + \omega$ while $\omega^3 = -2 - \omega$. As a final check of our theory, note that  indeed $\omega^6 = 2 + 5\omega \equiv 2$ mod $p$. 
\end{ex}

Now, before we can use this result to find the elasticity of $\bbz[p\omega]$, we must develop some machinery unique to quadratic rings of integers.

\begin{lem}
    Let $\bbz[\omega]$ be a quadratic ring of integers and $p$ a rational prime. If $\pi_1 \in \bbz[p\omega]$ and $\pi_2  \in \bbz[\omega] \,\backslash\, \bbz[p\omega]$, then $\pi_1\pi_2 \in \bbz[p\omega] \text{ if and only if } \pi_1 \in p\bbz[\omega]$.
\end{lem}

\begin{proof}

Let $\pi_1 = a + pb\omega$ and $\pi_2 = c + d\omega$. Then $\pi_1 \pi_2 = ac + px + (ad + py)\omega \in \bbz[p\omega]$ if and only if $p \vert ad$ if and only if $p \vert a$ or $p \vert d$. Equivalently, $\pi_1 \pi_2 \in \bbz[p\omega]$ if and only if $a + pb\omega = \pi_1 \in p\bbz[\omega]$ or $c+d\omega = \pi_2 \in \bbz[p\omega]$. The result then follows from the fact that $\pi_2 \notin \bbz[p\omega]$. 
    
\end{proof}

As we know, for any order $\MO \subseteq \MO_k$, the elements of the conductor are those $\alpha \in \MO_K$ such that $\alpha \MO_K \subseteq \MO$. That is, the elements of the conductor \enquote{pull} elements of the ring of integers down into the order. Lemma 3.4 tells us that for $\bbz[p\omega] \subseteq \bbz[\omega]$ as above, elements of $\MO$ outside the conductor $\textit{never}$ do this. We will now leverage this fact to demonstrate a nice property of factorizations in $\bbz[p\omega]$. 

\begin{thm}
    Let $\bbz[\omega]$ be a quadratic ring of integers with class number 1 and $\beta \in \bbz[p\omega]$ with prime factorization $p^n\pi_1 \dotsm \pi_k$ such that $(\pi_i,p) = 1$ in $\bbz[\omega]$. For a fixed $j$, assume $\pi_j \in \bbz[p\omega]$. Then, if $\alpha_1 \dotsm \alpha_m$ is an irreducible factorization of $\beta$ in $\bbz[p\omega]$, $\pi_j = \alpha_t$ for some  $1 \leq t \leq m$. 
    \end{thm}

\begin{proof}
    Without loss of generality we will prove the statement for $\pi_1 \in \bbz[p\omega]$. Assume $\pi_1 \in \bbz[p\omega]$. Note, $\pi_1 \notin p\bbz[\omega]$ as $(\pi_1,p) = 1$.

    Let $\beta = \alpha_1 \alpha_2 \dotsm \alpha_n$ be an irreducible factorization of $\beta$ in $\bbz[p\omega]$. As $\pi_1$ is prime and $\pi_1 | \beta$,  without loss of generality $\pi_1 | \alpha_1$ as an element of $\bbz[\omega]$. Thus, by the uniqueness of the prime factorization of $\beta$ in $\bbz[\omega]$, $\alpha_1 = \pi_1\cdot (p^r\pi_{i_1} \dotsm \pi_{i_l})u$  where $u \in U(\bbz[\omega]) = U(\bbz[p\omega])=\{\pm1\}$. If $r \geq 1$, then $(p^r \pi_{i_1} \dotsm \pi_{i_l})u \in \bbz[p\omega]$, and thus $\alpha_1$ is not irreducible. Hence, we assume $r = 0$. So we consider $\alpha_1 = \pi_1\cdot (\pi_{i_1} \dotsm \pi_{i_l})u$. If $(\pi_{i_1} \dotsm \pi_{i_l})u \notin \bbz[p\omega]$, by Lemma 3.4, $\pi_1 \notin p\bbz[i]$ implies $\pi_1 \cdot (\pi_{i_1} \dotsm \pi_{i_l})u = \alpha_1 \notin \bbz[pi]$, a contradiction. 
    
    So, $(\pi_{i_1} \dotsm \pi_{i_l})u \in \bbz[p\omega]$, but $\alpha_1 = \pi_1\cdot (\pi_{i_1} \dotsm \pi_{i_l})u$ is irreducible in $\bbz[p\omega]$, so $(\pi_{i_1} \dotsm \pi_{i_l})u \in U(\bbz[p\omega])$. Thus, $\pi_1 = \alpha_1$ up to a unit, and $\pi_1$ appears in every irreducible factorization of $\beta$ in $\bbz[p\omega]$. 
\end{proof}

Now, to determine the elasticity of $\bbz[p\omega]$, we need only consider elements in the conductor $p\bbz[\omega]$. This is because ideals relatively prime to the conductor retain unique factorization into prime ideals. In particular, principal ideals generated by elements relatively prime to the conductor. In our case, the elements outside the conductor. Thus, these elements behave in the order as they would in a ring of integers. That is, their elasticity is bounded above by $\frac{D(Cl(\bbz[p\omega]))}{2} \leq \frac{|Cl(\bbz[p\omega])|}{2} = \frac{(p+1)\cdot|Cl(\bbz[\omega])|}{2} = \frac{p+1}{2}$ when $\binom{d}{p}=-1$ (Theorem 2.2). As we will soon see, elements in the conductor will achieve an upper bound which exceeds this value. Theorem 3.5 allows us to narrow our search even further.

Take any $\gamma \in p\bbz[\omega]$ with prime factorization $\gamma = p^r \pi_1 \pi_2 \dotsm \pi_m$ in $\bbz[\omega]$, where $r \geq 1$ and $(p, \pi_i) = 1$. If $\pi_1 \in \bbz[p\omega]$, let $\delta = p^r \pi_2\dotsm \pi_m \in p\bbz[\omega]$. We may write $\rho(\delta) = \frac{n_1}{n_2}$ where $n_1$ and $n_2$ are the lengths of the longest and shortest irreducible factorizations of $\beta$ in $\bbz[p\omega]$ respectively. Note, the existence of $n_1$ is guaranteed by \cite{halter1995elasticity}. Theorem 3.5, together with the fact that $\bbz[p\omega]$ is a domain, implies that every irreducible factorization of $\gamma$ is of the form $\pi_1\delta_1\dotsm \delta_t$ where $\delta_1\dotsm \delta_t$ is an irreducible factorization of $\delta$. Thus, $\rho(\gamma) = \frac{n_1 + 1}{n_2 + 1} \leq \frac{n_1}{n_2} = \rho(\delta)$.

So for any element $p^r \pi_1 \pi_2 \dotsm \pi_m$ in the conductor with $\pi_i \in \bbz[p\omega]$, we can find an element $p^r \hat{\pi}_1 \hat{\pi}_2 \dotsm \hat{\pi}_k$ in the conductor with greater elasticity and $\hat{\pi}_i \notin \bbz[p\omega]$ for all $1 \leq i \leq k$. Therefore, to determine $\rho(\bbz[p\omega])$, we need only consider elements of the latter form. We are now prepared to prove the ultimate result of this section.

\begin{thm}
    Let $\bbq(\sqrt{d})$ be a quadratic number field with class number $1$ and $d < -3$. If $p$ is a rational prime such that $\binom{d}{p} = -1$, $\rho(\bbz[p\omega]) = 1 + \frac{p}{2}$.
\end{thm}

\begin{proof}
    Let $\alpha \in \bbz[\omega]$ with prime factorization $\alpha = p^k \pi_1 \pi_2 \dotsm \pi_m$ in $\bbz[\omega]$ where $(\pi_i, p) = 1$. Recall that for $\beta \notin p\bbz[\omega] = \mathfrak{f}$, we have 
    \[
    \rho(\beta) \leq \frac{D(Cl(\bbz[p\omega]))}{2} \leq \frac{|Cl(\bbz[p\omega])|}{2}=\frac{p+1}{2} < 1 + \frac{p}{2}.
    \]

    Thus, we may assume $\alpha \in p\bbz[\omega]$ ($k \geq 1$). Furthermore, as noted above, theorem 3.5 implies that we may also assume $\pi_i \notin \bbz[p\omega]$ for all $1 \leq i \leq k$. Hence, the finest possible irreducible factorization of $\alpha$ in $\bbz[p\omega]$ would have the form $p^k(\pi_1 \pi_2) \dotsm (\pi_{m-1}\pi_m)$ if $m$ is even, and $p^k(\pi_1 \pi_2) \dotsm (\pi_{m-2} \pi_{m-1}\pi_m)$ if $m$ is odd. That is, a factorization of length $k + \lfloor \frac{m}{2} \rfloor$.

    \underline{Case 1}: $m < kp$. We know that every irreducible factor of $\alpha$ can contain at most one factor of $p$. Thus, every irreducible factorization has length at least $k$, so 
    \[
    \rho(\alpha) \leq \frac{k + \lfloor \frac{m}{2} \rfloor}{k} \leq \frac{k + \frac{m}{2}}{k} < \frac{k + \frac{kp}{2}}{k} = 1 + \frac{p}{2}.
    \]
    
    \underline{Case 2}: $m \geq kp$. Now, any irreducible factorization of $\alpha$ must take the form 
    \[
    (p \pi_{1,1} \dotsm \pi_{1,n_1})\dotsm (p \pi_{k,1} \dotsm \pi_{k,n_k}) ( \pi_{k+1,1} \dotsm \pi_{k+1,n_{k+1}}) \dotsm ( \pi_{t,1} \dotsm \pi_{t,n_t}) 
    \]
    By Lemma 2.1 and Theorem 3.2, $n_i \leq p$ for $1 \leq i \leq k$. Also, for all $k + 1 \leq i \leq t$, $\beta = ( \pi_{j,1} \dotsm \pi_{j,n_{j}}) ( \Bar{\pi}_{j,1} \dotsm \Bar{\pi}_{j,n_j}) \in \bbz[p\omega] \backslash p\bbz[\omega]$ where $\Bar{\pi}$ is the complex conjugate of $\pi$. Once again, this implies $\rho(\beta) \leq \frac{p + 1}{2}$. Using the factorizations $( \pi_{j,1} \dotsm \pi_{j,n_{j}}) ( \Bar{\pi}_{j,1} \dotsm \Bar{\pi}_{j,n_j}) = (\pi_{j,1} \Bar{\pi}_{j,1}) \dotsm (\pi_{j,n_j} \Bar{\pi}_{j,n_j})$, we find $\frac{n_j}{2} \leq \rho(\beta) \leq \frac{p+1}{2}$ which implies $ n_j \leq p + 1$ for $k+1 \leq j \leq t$. This is sufficient to determine that the length of any given factorization is at least $k + \lceil \frac{m-kp}{p+1} \rceil$. Hence,
    \[
    \rho(\alpha) \leq \frac{k + \lfloor \frac{m}{2} \rfloor}{k + \lceil \frac{m-kp}{p+1} \rceil} \leq \frac{k + \frac{m}{2}}{k + \frac{m-kp}{p+1}} = \frac{2k + m}{2k + \frac{2m - 2kp}{p+1}} = (p+1)\frac{2k+m}{2k+2m} 
    \]
     which is maximized when $m$ is minimized. So $m \geq kp$ implies
    \[
    \rho(\alpha) \leq (p+1)\frac{2k+kp}{2k+2kp}=\frac{p+1}{2}\cdot\frac{2+p}{1+p} = 1 + \frac{p}{2}.
    \]
    Hence, we have determined $\rho(\bbz[p\omega]) \leq 1 + \frac{p}{2}$. 
    
    Finally, let $\gamma_1 \gamma_2 \dotsm \gamma_p$ be a $\bbz[\omega]$-product with no $\bbz[p\omega]$-subproduct. Note, this is guaranteed by Theorem 3.2. By Lemma 2.1, $p \gamma_1 \gamma_2 \dotsm \gamma_p$ is irreducible in $\bbz[p\omega]$. If we let $\delta = p^2 \gamma_1 \gamma_2 \dotsm \gamma_p \Bar{\gamma_1} \Bar{\gamma_2} \dotsm \Bar{\gamma_p}$, the factorizations   
    \[
    (p \gamma_1 \gamma_2 \dotsm \gamma_p)(p \Bar{\gamma_1} \Bar{\gamma_2} \dotsm \Bar{\gamma_p}) = p^2 (\gamma_1 \Bar{\gamma_1})(\gamma_2 \Bar{\gamma_2})\dotsm (\gamma_p \Bar{\gamma_p})
    \]
    show us that $\rho(\delta) \geq \frac{2 + p}{2} = 1 + \frac{p}{2}$, and consequently $\rho(\bbz[p\omega]) \geq 1 + \frac{p}{2}$ -- establishing the result. 
    
\end{proof}

First, we note that $\rho(\bbz[p\omega]) = 1 + \frac{p}{2} = \frac{\bar{D}(\MO)+1}{2}$. This is remarkably similar to the formula for elasticity of rings of integers given in Theorem 1.4---further motivating our approach. In particular, consider our construction of an element obtaining the upper bound and that of Narkiewicz (\cite{Narkiewicz1995Note}). From the proof, we can also see that $\bbz[p\omega]$ has \textit{accepted} elasticity. That is, there exists $x \in \bbz[p\omega]$ such that $\rho(x) = \rho(\bbz[p\omega])$. On top of this, we have a procedure for finding such an element. The orders which have accepted elasticity in rings of integers were first characterized by Halter-Koch (\cite{halter1995elasticity}) in 1995. He also gave an upper bound on said elasticities. This was the same year in which Narkiewicz (\cite{Narkiewicz1995Note}) used the Davenport constant to give an explicit formula for the elasticity of rings of integers. However, despite all this activity, we have not produced results on calculating elasticity of proper orders over the last 30 years. The generalized Davenport constant gave us a natural way to use the information in the rings of integers to answer this question and led us to a very constructive solution. It will be interesting to see where else this concept can be applied. 

\begin{ex}
    Theorem 3.6 affirms the previous claim that $\rho(\bbz[5 \cdot \frac{1 + \sqrt{-7}}{2}]) = \rho(\bbz[5\omega]) = 1 + \frac{5}{2} = \frac{7}{2}$. Recalling that $\omega^5$ has no $\bbz[5\omega]$-subproduct, we see that $(5\omega^5)(5 \Bar{\omega}^5) = 5^2(\omega \Bar{\omega})^5 = 5^2 2^5$ are irreducible factorizations of the element $800$, so we have $\rho(800) = \frac{7}{2}$ as desired. 
\end{ex}

As noted, we assumed $d < -3$ to ensure the unit groups of the order and its corresponding ring of integers were the same. Of course, this only precludes two of the imaginary quadratic unique factorization domains. It turns out, while our analysis must change slightly, the approach we have taken is viable for these rings as well. 

\begin{ex}

Consider the Gaussian integers $\bbz[i]$ and the order $\bbz[19i]$ which have units $\{\pm 1,\pm i \}$ and $\{\pm1\}$ respectively. Note that the prime $2 + 3i$ generates $\bbf_{19}[i]^*$, so as an element of $\bbz[i]$, $(2+3i)^{19}$ has no $\mathcal{O}$-subproduct. However, $19\cdot (2+3i)^{19}$ is not irreducible in $\bbz[19i]$ as Lemma 2.1 might suggest at first glance. This is because, considering the coefficients mod $19$, $(2 + 3i)^{10} \in i\bbf_{19}$, so $19\cdot (2+3i)^{19}$ factors non-trivially in $\bbz[19i]$ as $[i(2+3i)^{10}]\cdot[19(2+3i)^9(-i)]$. This happens because $\pm i$ are units in $\bbz[i]$ but not $\bbz[19i]$. The same can be said for any primitive element $\gamma$ because $\gamma^{10}$ will always be in $i\bbf_{19}^*$ as this subgroup contains all elements of order 2. Regardless, with some minor modifications to the proof above, one can show that $\rho(\bbz[pi]) = \frac{p+3}{4}$ for $p \equiv 3$ mod $4$.

\end{ex}

Finally, this approach continues to bear fruit as we move to consider the case where the ring of integers $\MO_K$ is a non-UFD. While we lose unique factorization for elements, we retain it for ideals. This would suggest that a more comprehensive analysis would require a more ideal-theoretic approach. However, as the following example illustrates, the generalized Davenport constant remains a vital tool for constructing elements of large elasticity without unique factorization. 

 \begin{ex}
     Consider the ring of integers $\bbz[\sqrt{-14}]$ and the order $\bbz[11\cdot \sqrt{-14}]$ $(p=11)$. It is easy to verify that the class group of $\bbz[\sqrt{-14}]$ is cyclic of order $4$. Thus, we have $D(Cl(\bbz[\sqrt{-14}])) = 4$. In keeping with the theory of \cite{Narkiewicz1995Note}, we would like to find a maximal zero-sum sequence in $Cl(\bbz[\sqrt{-14}])$. Allowing $(19, 9 + \sqrt{-14}) = \mathcal{P}$, we see that $[\mathcal{P}]^4$ is one such sequence. Now, $\mathcal{P}^4 = (325 + 42\sqrt{-14})$, and  we note also that the image of $325 + 42\sqrt{-14}$ under the map defined in Theorem 3.2 is a generator of $\bbf_{11}[\hat{\omega}]^*$. Notably, this implies that $(325 + 42 \sqrt{-14})^{11}$ has no $\bbz[11\cdot \sqrt{-14}]$-subproduct. Therefore, the element $[11\cdot (325 + 42\sqrt{-14})^{11}][11\cdot (325 - 42\sqrt{-14})^{11}] = 11^2\cdot 19^{44}$ has elasticity at least $\frac{2 + 44}{2} = 1 + 11\cdot \frac{4}{2} = 1 + 11\cdot \frac{D(Cl(\bbz[\sqrt{-14}]))}{2} = 1 + 11\cdot \rho(\bbz[\sqrt{-14}])$ in $\bbz[11 \cdot \sqrt{-14}]$.
 \end{ex}   

Using ideal-theoretic versions of some of the techniques developed in this paper, it has been subsequently shown by Moles and the author of the current paper that the elasticity of $\bbz[11\cdot\sqrt{-14}]$ is $\frac{D(Cl(\bbz[11\cdot \sqrt{-14}]))}{2}$. However, the construction of an element of large elasticity in Example 3.9 remains integral as it is fundamentally different from those which require knowledge of the structure of $Cl(\bbz[11\cdot\sqrt{-14}])$. In fact, it allows not only for the exact calculation of $\rho(\bbz[11\cdot\sqrt{-14}])$, but in so doing also determines the structure of $Cl(\bbz[11\cdot\sqrt{-14}])$. For further details, see \cite{kettinger2025elasticity}.

\bibliography{DavConstant}{}
\bibliographystyle{plain}

\end{document}